\documentclass[a4paper,12pt]{article}
\usepackage{amsmath,amsfonts,amssymb,amsthm,graphicx,enumerate,mathrsfs,color}

\newtheorem{definition}{Definition}
\newtheorem{theorem}{Theorem}
\newtheorem{lemma}[theorem]{Lemma}
\newtheorem{proposition}[theorem]{Proposition}

\newtheorem{remark}[theorem]{Remark}

\newcommand{\ten}[1]{\mathcal{#1}}
\newcommand{\mt}{$\mathcal{M}$}
\newcommand{\zt}{$\mathcal{Z}$}
\newcommand{\htt}{$\mathcal{H}$}

\title{\mt-Tensors and Nonsingular \mt-Tensors}
\author{
Weiyang Ding\thanks{E-mail: 11110180009@fudan.edu.cn. School of
Mathematical Sciences, Fudan University, Shanghai, 200433, P. R. of
China. This author is supported by 973 Program Project under grant
2010CB327900.} \and Liqun Qi\thanks{E-mail: liqun.qi@polyu.edu.hk.
Department of Applied Mathematics, The Hong Kong Polytechnic
University, Hong Kong. This author is supported by the Hong Kong
Research Grant Council under grants (No. PolyU 501909, 502510,
502111 and 501212).} \and Yimin Wei\thanks{Corresponding author (Y.
Wei). E-mail: ymwei@fudan.edu.cn and yimin.wei@gmail.com. School of
Mathematical Sciences and Shanghai Key Laboratory of Contemporary
Applied Mathematics, Fudan University, Shanghai, 200433, P. R. of
China. This author is supported by the National Natural Science
Foundation of China under grant 11271084.} }
\date{\today}

\begin{document}

\maketitle

\begin{abstract}
  The \textit{M}-matrix is an important concept in matrix theory, and has many applications.
  Recently, this concept has been extended to higher order tensors \cite{Zhang12}.
  In this paper, we establish some important properties of \mt-tensors and nonsingular \mt-tensors. An \mt-tensor is a \zt-tensor.
  We show that a \zt-tensor is a nonsingular \mt-tensor if and only if it is semi-positive. Thus, a nonsingular \mt-tensor has all positive diagonal entries; and an \mt-tensor,
  regarding as the limitation of a series of nonsingular \mt-tensors, has all nonnegative diagonal entries. We introduce even-order monotone  tensors and present their spectral properties.
  In matrix theory, a \textit{Z}-matrix is a nonsingular \textit{M}-matrix if and only if it is monotone. This is no longer true in the case of higher order tensors. We show that an even-order
   monotone \zt-tensor is an even-order nonsingular \mt-tensor but not vice versa. An example of an even-order nontrivial monotone \zt-tensor is also given.

  \bigskip


  \bigskip

\end{abstract}

\newpage

\section{Introduction}

Several topics in multi-linear algebra have attracted considerable attention in recent years, especially on the eigenvalues of tensors \cite{Qi05} and  higher order tensor decompositions \cite{Kolda09}.
Tensors (or hypermatrices) generalize the concept of matrices in linear algebra. The main difficulty in tensor problems is that they are generally nonlinear. Therefore large amounts of results for matrices are
never pervasive for higher order tensors. However there
are still some results preserved in the case of higher order tensors.

As an important example, some properties of spectra (or eigenvalues) in linear algebra remain true to tensors. Qi generalizes the concept of eigenvalues to higher order tensors in \cite{Qi05}
by defining the tensor-vector product as
$$
(\ten{A}x^{m-1})_i = \sum_{i_2,i_3,\cdots,i_m = 1}^n a_{i i_2 \cdots
i_m} x_{i_2} \cdots x_{i_m},
$$
where a multi-array $\ten{A}$ is an $m$-order $n$-dimensional tensor in $\mathbb{C}^{n \times n \times \cdots \times n}$ and $x$ is a vector in $\mathbb{C}^n$. We call $\lambda$ as an eigenvalue of tensor $\ten{A}$,
if there exists a nonzero vector $x \in \mathbb{C}^n$ such that
$$
\ten{A}x^{m-1} = \lambda x^{[m-1]},
$$
where $x^{[m-1]} = (x_1^{m-1},x_2^{m-1},\cdots,x_n^{m-1})^{\rm T}$ denotes the componentwise $(m-1)$-th power of $x$. Further, we call $\lambda$ as an H-eigenvalue, H$^+$-eigenvalue, or H$^{++}$-eigenvalue if $x \in \mathbb{R}^n$,
$x \in \mathbb{R}_+^n$ ($x \geq 0$), or $x \in \mathbb{R}_{++}^n$ ($x > 0$), respectively. Also Qi \cite{Qi05} introduces another kind of eigenvalues for higher order tensors. We call $\lambda$ an E-eigenvalue of tensor $\ten{A}$,
if there exists a nonzero vector $x \in \mathbb{C}^n$ such that
$$
\ten{A}x^{m-1} = \lambda x,\quad  x^{\rm T} x = 1;
$$
and we call $\lambda$ a Z-eigenvalue if $x \in \mathbb{R}^n$. There have been extensive studies and applications of both kinds of eigenvalues for tensors.

\textit{M}-matrices are an important class of matrices and have been well studied (cf. \cite{Berman94}). They are closely related with spectral graph theory, the stationary distribution of Markov chains and the
 convergence of iterative methods for linear systems. Zhang et al. extended \textit{M}-matrices to \mt-tensors in \cite{Zhang12} and studied their properties. The main result in their paper is that every eigenvalue of
 an \mt-tensor has a positive real part, which is the same as the \textit{M}-matrix.

The main motivation of this paper is that there are no less than
fifty equivalent definitions of nonsingular \textit{M}-matrices
\cite{Berman94}. We extend the other two definitions of nonsingular
\textit{M}-matrices, semi-positivity and monotonicity
\cite{Berman94}, to higher order tensors. How can they be
generalized into higher order case? Are they still the equivalent
conditions of higher order nonsingular \mt-tensors?
 What properties of \mt-tensors can they  indicate? We will answer these questions in our paper.

In this way, we obtain some important properties of \mt-tensors and
nonsingular \mt-tensors. An \mt-tensor is a \zt-tensor. We prove
that a \zt-tensor is a nonsingular \mt-tensor if and only if it is
semi-positive. Thus, a nonsingular \mt-tensor has all positive
diagonal entries; and an \mt-tensor, regarding as the limitation of
a series of nonsingular \mt-tensors, has all nonnegative diagonal
entries. We introduce even-order monotone  tensors and establish
their spectral properties. In matrix theory, a \textit{Z}-matrix is
a nonsingular \textit{M}-matrix if and only if it is monotone
\cite{Berman94}. It is no longer true in the case of higher order
tensors. We show that an even-order monotone \zt-tensor is an
even-order nonsingular \mt-tensor but not vice versa. An example of
an even-order nontrivial monotone \zt-tensor is also given.

An outline of this paper is as follows. Some preliminaries about
tensors and \textit{M}-matrices are presented in Section 2. We
investigate semi-positive \zt-tensors and monotone \zt-tensors in
Sections 3 and 4, respectively. We discuss \htt-tensors, an
extension of \mt-tensors, in Section 5. Finally, we draw some
conclusions in the last section.


\section{Preliminaries}

We present some preliminaries about the Perron-Frobenius theorem for nonnegative tensors and \textit{M}-matrices.

\subsection{Nonnegative Tensor}

Because of the difficulties in studying the properties of a general
tensor, researchers focus on some \emph{structured} tensors. The
nonnegative tensor is one of the most well studied tensors. A tensor
is said to be nonnegative, if all its entries are nonnegative.

The Perron-Frobenius theorem is the most famous result for
nonnegative matrices (cf. \cite{Berman94}), which investigates the
spectral radius of nonnegative matrices. Researchers also propose
the similar results for nonnegative tensors, and refer them as the
Perron-Frobenius theorem for nonnegative tensors. This theorem also
studies the spectral radius of a nonnegative tensor $\ten{B}$,
$$
\rho(\ten{B}) = \max\big\{|\lambda| \,\big|\, \lambda \text{ is an
eigenvalue of } \ten{B} \big\}.
$$

Before stating the Perron-Frobenius theorem, we briefly introduce
the conceptions of irreducible and weakly irreducible tensors.
\begin{definition}[Irreducible Tensor \cite{Chang08}]
A tensor $\ten{B}$ is called \emph{reducible}, if there exists a non-empty proper index subset $I \subset \{1,2,\cdots,n\}$ such that
$$
b_{i_1 i_2 \cdots i_m} = 0,\ \forall i_1 \in I,\ \forall i_2, i_3,\cdots,i_m \notin I.
$$
Otherwise, we say A is \emph{irreducible}.
\end{definition}
\begin{definition}[Weakly Irreducible Nonnegative Tensor \cite{Friedland13}]
We call a nonnegative matrix $GM(\ten{B})$ the \emph{representation}
associated to a nonnegative tensor $\ten{B}$, if the $(i,j)$-th
entry of $GM(\ten{B})$ is defined to be the summation of $b_{i i_2
i_3\cdots i_m}$ with indices $\{i_2,i_3,\cdots,i_m\} \ni j$. We call
a tensor $\ten{B}$ \emph{weakly reducible}, if its representation
$GM(\ten{B})$ is reducible. If $\ten{B}$ is not weakly reducible,
then it is called \emph{weakly irreducible}.
\end{definition}

Now with these conceptions, we can recall several results of the
Perron-Frobenius theorem for nonnegative tensors that we will use in
this paper.
\begin{theorem}[The Perron-Frobenius Theorem for Nonnegative
Tensors]\label{P-F}
If $\ten{B}$ is a nonnegative tensor of order
$m$ and dimension $n$, then $\rho(\ten{B})$ is an eigenvalue of
$\ten{B}$ with a nonnegative eigenvector $x \in \mathbb{R}_+^n$
\emph{(\cite{Yang10})}.

If furthermore $\ten{B}$ is strictly nonnegative, then $\rho(\ten{B})>0$ \emph{(\cite{Hu11})}.

If furthermore $\ten{B}$ is weakly irreducible, then $\rho(\ten{B})$ is an eigenvalue of $\ten{B}$ with a positive eigenvector $x \in \mathbb{R}_{++}^n$ \emph{(\cite{Friedland13})}.

Suppose that furthermore $\ten{B}$ is irreducible. If $\lambda$ is an eigenvalue with a nonnegative eigenvector, then $\lambda = \rho(\ten{B})$ \emph{(\cite{Chang08})}.
\end{theorem}

\subsection{\textit{M}-Matrix}\label{mmatrix}

\textit{M}-matrix arises frequently in scientific computations, and
we briefly introduce its definition and properties in this
section.\\

A matrix is called a \emph{\textit{Z}-matrix} if all its
off-diagonal entries are non-positive. It is apparent that a
\textit{Z}-matrix $A$ can be written as  \cite{Berman94}
$$
A = sI-B,
$$
where $B$ is a nonnegative matrix ($B \geq 0$) and $s > 0$; When $s
\geq \rho(B)$, we call $A$ as an \emph{\textit{M}-matrix}; And
further when $s > \rho(B)$, we call $A$ as a \emph{nonsingular
\textit{M}-matrix} \cite{Berman94}.\\

There are more than fifty conditions in the literature that are equivalent to the definition of nonsingular \textit{M}-matrix.
We just list eleven of them here, which will be involved in our paper (cf. \cite{Berman94}). \\

If $A$ is a \textit{Z}-matrix, then the following conditions are equivalent:

\begin{enumerate}[({C}1)]
  \item $A$ is a nonsingular \textit{M}-matrix;
  \item $A+D$ is nonsingular for each nonnegative diagonal matrix $D$;
  \item Every real eigenvalue of $A$ is positive;
  \item $A$ is \emph{positive stable}; that is, the real part of each eigenvalue of $A$ is positive;
  \item $A$ is \emph{semi-positive}; that is, there exists $x > 0$ with $Ax > 0$;
  \item There exists $x \geq 0$ with $Ax > 0$;
  \item $A$ has all positive diagonal entries and there exists a positive diagonal matrix $D$ such that $AD$ is strictly diagonally dominant;
  \item $A$ has all positive diagonal entries and there exists a positive diagonal matrix $D$ such that $DAD$ is strictly diagonally dominant;
  \item $A$ is \emph{monotone}; that is, $Ax \geq 0$ implies $x \geq 0$;
  \item There exists an inverse-positive matrix $B$ and a nonsingular \textit{M}-matrix $C$ such that $A=BC$;
  \item $A$ has a \emph{convergent regular splitting}; that is, $A$ has a representation $A=M-N$, where $M^{-1} \geq 0$, $N \geq 0$, and $\rho(M^{-1}N)<1$;
  \item $\cdots \cdots \cdots$
\end{enumerate}

\subsection{\mt-Tensor}

Zhang et al. defined the \mt-tensor following the definition of
\textit{M}-matrix \cite{Zhang12}. In this section, we will introduce
their results for \mt-tensors. First, they define the \zt-tensor and
\mt-tensor as follows. We call a tensor the unit tensor and denote
it $\ten{I}$, if all of its diagonal entries are 1 and all of its
off-diagonal entries are 0.

\begin{definition}[\zt-tensor]
We call a tensor $\ten{A}$ as a \emph{\zt-tensor}, if all of its
off-diagonal entries are non-positive, which is equivalent to write
$\ten{A} = s \ten{I} - \ten{B}$, where $s>0$ and $\ten{B}$ is a
nonnegative tensor ($\ten{B} \geq 0$).
\end{definition}

\begin{definition}[\mt-tensor]
We call a \zt-tensor $\ten{A} = s \ten{I} - \ten{B}$ ($\ten{B} \geq 0$) as an \emph{\mt-tensor} if $s \geq \rho(\ten{B})$; We call it as a \emph{nonsingular \mt-tensor} if $s > \rho(\ten{B})$.
\end{definition}

The main results in their paper show two equivalent conditions for the definition of nonsingular \mt-tensor, which are the extensions of the conditions (C3) and (C4). If $\ten{A}$ is a \zt-tensor,
then the following conditions are equivalent:
\begin{enumerate}[({D}1)]
  \item $\ten{A}$ is a nonsingular \mt-tensor;
  \item Every real eigenvalue of $\ten{A}$ is positive;
  \item The real part of each eigenvalue of $\ten{A}$ is positive.
\end{enumerate}

\subsection{Notation}

We adopt the following notation in this paper. The calligraphy
letters $\mathcal{A},\mathcal{B},\mathcal{D},\ldots$ denote the
tensors; the capital letters $A,B,D,\ldots$ represent the matrices;
the lowercase letters $a, x, y,\ldots$ refer to the vectors; and the
Greek letters $\alpha,\beta,\lambda,\ldots$ designate the scalars.
Usually, the tensors in our paper are of order $m$ and dimension
$n$. When we write $\mathcal{A} \geq 0$, $A \geq 0$, or $x \geq 0$,
we mean that every entry of $\mathcal{A}$, $A$, or $x$ is
nonnegative; when we write $\mathcal{A} > 0$, $A > 0$, or $x > 0$,
we mean that every entry of $\mathcal{A}$, $A$, or $x$ is positive.

The product of a tensor $\ten{A} \in \mathbb{R}^{n \times n \times \cdots \times n}$ and a matrix $X \in \mathbb{R}^{n \times n}$ on mode-$k$ \cite{Kolda09} is defined as
$$
(\ten{A} \times_k X)_{i_1 \cdots j_k \cdots i_m} = \sum_{i_k = 1}^{n} a_{i_1 \cdots i_k \cdots i_m} x_{i_k j_k};
$$
then denote
$$
\ten{A}X^{m-1} = \ten{A} \times_2 X \cdots \times_m X,\text{ and } X_1^{\rm T} \ten{A} X_2^{m-1} = \ten{A} \times_1 X_1 \times_2 X_2 \cdots \times_m X_2;
$$
Finally, we introduce the ``composite'' of a diagonal tensor $\ten{D}$ and another tensor $\ten{A}$ as
$$
(\ten{D} \ten{A})_{i_1 i_2 \cdots i_m} = d_{i_1 i_1 \cdots i_1} \cdot a_{i_1 i_2 \cdots i_m},
$$
which indicates that $(\ten{D} \ten{A}) x^{m-1} =
\ten{D}\Big((\ten{A}
x^{m-1})^{\left[\frac{1}{m-1}\right]}\Big)^{m-1}$; And the
``inverse'' of a diagonal tensor $\ten{D}$ as \cite{Kolda09},
$$
(\ten{D}^{-1})_{i_1 i_1 \cdots i_1} = d_{i_1 i_1 \cdots i_1}^{-1}, \text{ and otherwise 0},
$$
which indicates that $(\ten{D}^{-1} \ten{D}) x^{m-1} = (\ten{D} \ten{D}^{-1}) x^{m-1} = \ten{I} x^{m-1} = x^{[m-1]}$.


\section{Semi-Positivity and Semi-Nonnegativity}
In this section, we will propose an equivalent definition of nonsingular \mt-tensors following the conditions (C5) and (C6) in Section \ref{mmatrix}.

\subsection{Definitions}

First, we extend the \emph{semi-positivity} \cite{Berman94} from
matrices to tensors.

\begin{definition}[Semi-positive tensor]
We call a tensor $\ten{A}$ as a \emph{semi-positive tensor}, if there exists $x > 0$ such that $\ten{A} x^{m-1} > 0$.
\end{definition}

Because of the continuity of the tensor-vector product on the
entries of the vector, the requirement $x > 0$ in the first
definition can be relaxed into $x \geq 0$. We verify this statement
as follows.

\begin{theorem}
A tensor $\ten{A}$ is semi-positive if and only if there exists $x \geq 0$ such that $\ten{A} x^{m-1} > 0$.
\end{theorem}
\begin{proof}
Define a map $T_\ten{A}(x) = (\ten{A}x^{m-1})^{\left[\frac{1}{m-1}\right]}$, then $x \mapsto T_\ten{A}$ is continuous and bounded \cite{Chang11}.

If $\ten{A}$ is semi-positive, then it is trivial that there is $x \geq 0$ such that $\ten{A} x^{m-1} > 0$ according to the definition.

If there exists $x \geq 0$ with $\ten{A} x^{m-1} > 0$, then there
must be a closed ball $B\big(T_\ten{A}(x),\varepsilon\big)$ in
$\mathbb{R}_{++}^n$, where $B(c,r) := \{ v \, \big|\, \|v-c\| \leq r
\}$. Since $T_\ten{A}$ is continuous, there exists $\delta>0$ such
that $T_\ten{A}(y) \in B\big(T_\ten{A}(x),\varepsilon\big)$ for all
$y \in B(x,\delta)$. Let $d$ be a zero-one vector with $d_i = 1$ if
$x_i = 0$ and $d_i = 0$ if $x_i > 0$. Take $y = x +
\frac{\delta}{\|d\|}d \in B(x,\delta)$. Then $y > 0$ and
$T_\ten{A}(y) > 0$. Therefore $\ten{A}$ is semi-positive.
\end{proof}

It is well known that a \textit{Z}-matrix is a nonsingular
\textit{M}-matrix if and only if it is semi-positive
\cite{Berman94}. Furthermore, we will come to a similar  conclusion
for nonsingular \mt-tensors.

\begin{theorem}\label{thm1}
A \zt-tensor is a nonsingular \mt-tensor if and only if it is semi-positive.
\end{theorem}

The proof of this theorem will be presented at the end of this
section after studying some properties of semi-positive \zt-tensors.

\subsection{Semi-Positive \zt-Tensors}

The first property is about the diagonal entries of a semi-positive \zt-tensor.
\begin{proposition}\label{prop2}
A semi-positive \zt-tensor has all positive diagonal entries.
\end{proposition}
\begin{proof}
When $\ten{A}$ is a semi-positive \zt-tensor, there exists $x > 0$ such that $\ten{A}x^{m-1} > 0$. Consider $\ten{A}x^{m-1}$, we have
$$
(\ten{A} x^{m-1})_i = a_{i i \cdots i} x_i^{m-1} + \sum_{(i_2,i_3,\cdots,i_m) \neq (i,i,\cdots,i)} a_{i i_2 \cdots i_m} x_{i_2} \cdots x_{i_m} > 0,
$$
for $i = 1,2,\cdots,n$. From $x_j > 0$ and $a_{i i_2 i_3\cdots i_m} \leq 0$ for $(i_2,i_3,\cdots,i_m) \neq (i,i,\cdots,i)$, we can conclude that $a_{i i \cdots i} > 0$ for $i = 1,2,\cdots,n$.
\end{proof}

Moreover we have a series of equivalent conditions of semi-positive \zt-tensors, following the conditions (C7), (C8), (C10), and (C11) in Section \ref{mmatrix}.

\begin{proposition}
A \zt-tensor $\ten{A}$ is semi-positive if and only if $\ten{A}$ has all positive diagonal entries and there exists a positive diagonal matrix $D$ such that $\ten{A}D^{m-1}$ is strictly diagonally dominant.
\end{proposition}
\begin{proof}
Let $D = \mathrm{diag}(d_1,d_2,\cdots,d_n)$. Then $\ten{A}D^{m-1}$ is strictly diagonally dominant means
$$
\big|a_{i i \cdots i} d_i^{m-1}\big| > \sum_{(i_2,i_3,\cdots,i_m)
\neq (i,i,\cdots,i)} \big|a_{i i_2 \cdots i_m} d_{i_2} \cdots
d_{i_m}\big|,\quad i = 1,2,\cdots,n.
$$

If $\ten{A}$ is a semi-positive \zt-tensor, then we know that $a_{i i \cdots i} > 0$ for $i = 1,2,\cdots,n$ from Proposition \ref{prop2}, $a_{i i_2 \cdots i_m} \leq 0$ for
 $(i_2,i_3,\cdots,i_m) \neq (i,i,\cdots,i)$, and there is $x>0$ with $\ten{A}x^{m-1}>0$. Let $D = \mathrm{diag}(x)$, we can easily conclude that $D$ is positive diagonal and $\ten{A}D^{m-1}$ is strictly diagonally dominant.

If $\ten{A}$ has all positive diagonal entries, and there exists a
positive diagonal matrix $D$ such that $\ten{A}D^{m-1}$ is strictly
diagonally dominant, let $x = \mathrm{diag}(D) > 0$, then
$\ten{A}x^{m-1}>0$ since $a_{i i \cdots i} > 0$ for $i =
1,2,\cdots,n$ and $a_{i i_2 \cdots i_m} \leq 0$ for $(i_2,i_3,
\cdots,i_m) \neq (i,i,\cdots,i)$. Thus $\ten{A}$ is a semi-positive
tensor.
\end{proof}

\begin{proposition}
A \zt-tensor $\ten{A}$ is semi-positive if and only if $\ten{A}$ has
all positive diagonal entries and there exist two positive diagonal
matrices $D_1$ and $D_2$ such that $D_1 \ten{A} D_2^{m-1}$ is
strictly diagonally dominant.
\end{proposition}
\begin{proof}
Notice that $D_1 \ten{A} D_2^{m-1}$ is strictly diagonally dominant if and only if $\ten{A}D_2^{m-1}$ is strictly diagonally dominant in sake of the positivity of $D_1$'s diagonal entries.
Therefore this proposition is a direct corollary of Proposition 5.
\end{proof}

\begin{proposition}
A \zt-tensor $\ten{A}$ is semi-positive if and only if there exists a positive diagonal tensor $\ten{D}$ and a semi-positive \zt-tensor $\ten{C}$ with $\ten{A} = \ten{D}\ten{C}$.
\end{proposition}\begin{proof}
Let $\ten{D}$ be the diagonal tensor of $\ten{A}$ and $\ten{C}=\ten{D}^{-1}\ten{A}$. Clearly, $\ten{A} = \ten{D}\ten{C}$.

If $\ten{A}$ is semi-positive \zt-tensor, then $\ten{D}$ is positive diagonal and there exists $x > 0$ with $\ten{A}x^{m-1} > 0$. Then the vector
$\ten{C}x^{m-1} = \ten{D}^{-1} \Big((\ten{A} x^{m-1})^{\left[\frac{1}{m-1}\right]}\Big)^{m-1}$ is also positive. So $\ten{C}$ is also a semi-positive \zt-tensor.

If $\ten{C}$ is a semi-positive \zt-tensor and $\ten{D}$ is positive diagonal, then there exists $x > 0$ with $\ten{C}x^{m-1} > 0$. Then the vector
$\ten{A}x^{m-1} = \ten{D} \Big((\ten{A} x^{m-1})^{\left[\frac{1}{m-1}\right]}\Big)^{m-1}$ is also positive. Thus $\ten{A}$ is a semi-positive \zt-tensor.
\end{proof}

\begin{remark} After we prove Theorem \ref{thm1}, Proposition 7 can be restated as: A \zt-tensor $\ten{A}$ is a nonsingular \mt-tensor if and only if there exists a positive diagonal tensor $\ten{D}$
and a nonsingular \mt-tensor $\ten{C}$ with $\ten{A} = \ten{D}\ten{C}$.
\end{remark}

\begin{proposition}
A \zt-tensor $\ten{A}$ is semi-positive if and only if there exists a positive diagonal tensor $\ten{D}$ and a nonnegative tensor $\ten{E}$ such that $\ten{A} = \ten{D} - \ten{E}$ and there exists $x>0$
with $(\ten{D}^{-1}\ten{E})x^{m-1} < x^{[m-1]}$.
\end{proposition}
\begin{proof}
Let $\ten{D}$ be the diagonal tensor of $\ten{A}$ and $\ten{E}=\ten{D}-\ten{A}$. Clearly, $\ten{A} = \ten{D}-\ten{E}$ and $\ten{D}^{-1}\ten{E} = \ten{I} - \ten{D}^{-1}\ten{A}$.

If $\ten{A}$ is a semi-positive \zt-tensor, then $\ten{D}$ is positive diagonal and there exists $x > 0$ with $\ten{A}x^{m-1} > 0$. Then $\ten{D}x^{m-1} > \ten{E}x^{m-1}$, and thus $(\ten{D}^{-1}\ten{E})x^{m-1} < x^{[m-1]}$.

If there exists $x>0$ with $(\ten{D}^{-1}\ten{E})x^{m-1} < x^{[m-1]}$, then $\ten{E}x^{m-1} < \ten{D}x^{m-1}$, and thus $\ten{A} x^{m-1} > 0$. Therefore $\ten{A}$ is a semi-positive \zt-tensor.
\end{proof}

\begin{remark} It follows from \cite[Lemma 5.4]{Yang10} that a semi-positive \zt-tensor can be splitted into $\ten{A} = \ten{D} - \ten{E}$,
where $\ten{D}$ is a positive diagonal tensor and $\ten{E}$ is a
nonnegative tensor with $\rho(\ten{D}^{-1}\ten{E}) < 1$.
\end{remark}

\subsection{Examples}

Next we  present some examples of nontrivial semi-positive and
semi-nonnegative \zt-tensors.

\begin{proposition}
A strictly diagonally dominant \zt-tensor with nonnegative diagonal
entries is a semi-positive \zt-tensor.
\end{proposition}
\begin{proof}
Use ${\bf e}$ to denote the all ones vector. It is direct to show
that a \zt-tensor $\ten{A}$ with all nonnegative diagonal entries is
strictly diagonally dominant is equivalent to $\ten{A}{\bf e}^{m-1}
> 0$. Since ${\bf e} > 0$, the result follows the definition of
semi-positive  \zt-tensors.
\end{proof}

\begin{proposition}\label{prop3}
A weakly irreducible nonsingular \mt-tensor is a semi-positive \zt-tensor.
\end{proposition}
\begin{proof}
When $\ten{A}$ is a nonsingular \mt-tensor, we write $\ten{A} = s
\ten{I} - \ten{B}$, where $\ten{B} \geq 0$ and $s > \rho(\ten{B})$.
Since $\ten{A}$ is weakly irreducible, so is $\ten{B}$. Then there
exists $x>0$ such that $\ten{B}x^{m-1} = \rho(\ten{B}) x^{[m-1]}$
from the Perron-Frobenius Theorem  for nonnegative tensors (cf.
Theorem \ref{P-F}), thus
$$
\ten{A}x^{m-1} = (s-\rho(\ten{B})) x^{[m-1]} > 0.
$$
Therefore $\ten{A}$ is a semi-positive tensor.
\end{proof}

\subsection{Proof of Theorem \ref{thm1}}

Our aim is to prove the equality relation between the following two sets:
$$
\{\text{semi-positive \zt-tensors}\} = \{\text{nonsingular \mt-tensors}\}.
$$
The first step is to verify the ``$\subseteq$'' part, which is relatively simple.

\begin{lemma}\label{lemma4}
A semi-positive \zt-tensor is also a nonsingular \mt-tensor.
\end{lemma}
\begin{proof}
When $\ten{A}$ is semi-positive, there exists $x > 0$ with $\ten{A} x^{m-1} > 0$. We write $\ten{A} = s \ten{I} - \ten{B}$ since $\ten{A}$ is a \zt-tensor, where $\ten{B} \geq 0$. Then
$$
\min_{1 \leq i \leq n} \frac{(\ten{B}x^{m-1})_i}{x_i^{m-1}} \leq \rho(\ten{B}) \leq \max_{1 \leq i \leq n} \frac{(\ten{B}x^{m-1})_i}{x_i^{m-1}}.
$$
Thus $s - \rho(\ten{B}) \geq s - \displaystyle\max_{1 \leq i \leq n} \frac{(\ten{B}x^{m-1})_i}{x_i^{m-1}} = \min_{1 \leq i \leq n} \frac{(\ten{A}x^{m-1})_i}{x_i^{m-1}} > 0$, since both $x$ and $\ten{A} x^{m-1}$ are positive.
Therefore $\ten{A}$ is a nonsingular \mt-tensor.
\end{proof}

The second step is to prove the ``$\supseteq$'' part employing a partition of general nonnegative tensors into weakly irreducible leading subtensors.

\begin{lemma}\label{lemma5}
A nonsingular \mt-tensor is also a semi-positive \zt-tensor; And an \mt-tensor is also a semi-nonnegative \zt-tensor.
\end{lemma}
\begin{proof}
Assume that a nonsingular \mt-tensor $\ten{A} = s\ten{I} - \ten{B}$ is weakly reducible, otherwise we have proved a weakly irreducible nonsingular \mt-tensor is also a semi-positive \zt-tensor in Proposition \ref{prop3}.
Then $\ten{B}$ is also weakly reducible. Following \cite[Theorem 5.2]{Hu11}, the index set $I = \{1,2,\cdots,n\}$ can be partitioned into $I = I_1 \sqcup I_2 \sqcup \cdots \sqcup I_k$ (here $A = A_1 \sqcup A_2$ means that
$A = A_1 \cup A_2$ and $A_1 \cap A_2 = \varnothing$) such that
\begin{enumerate}[(1)]
  \item $\ten{B}_{I_t I_t \cdots I_t}$ is weakly irreducible,
  \item $b_{i_1 i_2 \cdots i_m} = 0$ for $i_1 \in I_t$ and $\{i_2, i_3,\cdots,i_m\} \nsubseteq I_t \sqcup I_{t+1} \sqcup \cdots \sqcup
  I_k$, \quad  $t = 1,2,\cdots,k$.
\end{enumerate}
Without loss of generality, we can assume that
\[
\begin{split}
&I_1 = \{1,2,\cdots,n_1\}, \\
&I_2 = \{n_1+1,n_1+2,\cdots,n_2\}, \\
&\cdots \quad \cdots \quad \cdots \\
&I_k = \{n_{k-1}+1,n_{k-1}+2,\cdots,n\}.
\end{split}
\]

We introduce the following denotations
$$
\ten{B}_{(t,a)} := \ten{B}_{I_t I_{a_1} \cdots I_{a_{m-1}}}
$$
and
$$
\ten{B}_{(t,a)} z_a^{m-1} := \ten{B}_{(t,a)} \times_{a_1} z_{a_1}
\times \cdots \times_{a_{m-1}} z_{a_{m-1}},
$$
where $a$ is an index vector of length $m-1$ and $z_j$'s are column vectors. We also apply $\ten{B}[J]$ to denote the leading subtensor $(b_{i_1 i_2 \cdots i_m})_{i_j \in J}$, where $J$ is an arbitrary index set.
Since $s > \rho(\ten{B}) \geq \rho(\ten{B}[I_t])$, the leading subtensors $s \ten{I} - \ten{B}[I_t]$ are irreducible nonsingular \mt-tensors. Hence they are also semi-positive, that is, there exists $x_t > 0$
with $s x_t^{[m-1]} - \ten{B}[I_t] x_t^{m-1} > 0$ for all $t = 1,2,\cdots,k$.

Consider the leading subtensor $\ten{B}[I_1 \sqcup I_2]$ first. For all vector $z_1$ of length $n_1$ and $z_2$ of length $n_2-n_1$, we write
$$
\ten{B}[I_1 \sqcup I_2] \begin{bmatrix} z_1 \\ z_2 \end{bmatrix}^{m-1} =
\begin{bmatrix}
\ten{B}[I_1] z_1^{m-1} + \sum\limits_{a \neq (1,1,\cdots,1)} \ten{B}_{(1,a)} z_a^{m-1} \\
\ten{B}[I_2] z_2^{m-1}
\end{bmatrix},
$$
where the entries of $a$ only contains $1$ and $2$. Take $z_1 = x_1$ and $z_2 = \varepsilon x_2$, where $\varepsilon \in (0,1)$ satisfies
$$
\varepsilon \cdot \sum_{a \neq (1,1,\cdots,1)} \ten{B}_{(1,a)}
x_a^{m-1} < s x_1^{[m-1]} - \ten{B}[I_1] x_1^{m-1}.
$$
Since $\sum\limits_{a \neq (1,1,\cdots,1)} \ten{B}_{(1,a)} z_a^{m-1}
\leq \varepsilon \Big(\sum\limits_{a \neq (1,1,\cdots,1)}
\ten{B}_{(1,a)} x_a^{m-1}\Big)$, it can be ensured that
$\ten{B}[I_1] z_1^{m-1} + \sum\limits_{a \neq (1,1,\cdots,1)}
\ten{B}_{(1,a)} z_a^{m-1} < s z_1^{[m-1]}$. Therefore we obtain
$$
\begin{bmatrix} x_1 \\ \varepsilon x_2 \end{bmatrix} > 0
\quad\text{and}\quad
s \begin{bmatrix} x_1 \\ \varepsilon x_2 \end{bmatrix}^{[m-1]} \hspace{-5pt}- \ten{B} \begin{bmatrix} x_1 \\ \varepsilon x_2 \end{bmatrix}^{m-1} > 0,
$$
so $\ten{A}[I_1 \sqcup I_2] = s \ten{I} - \ten{B}[I_1 \sqcup I_2]$ is a semi-positive \zt-tensor.

Assume that $\ten{A}[I_1 \sqcup I_2 \sqcup \cdots \sqcup I_t]$ is a
semi-positive \zt-tensor. Consider the leading subtensor
$\ten{A}[I_1 \sqcup I_2 \sqcup \cdots \sqcup I_{t+1}]$ next.
Substituting the index sets
 $I_1$ and $I_2$ above with $I_1 \sqcup I_2 \sqcup \cdots \sqcup I_t$ and $I_{t+1}$, respectively, we can conclude that $\ten{A}[I_1 \sqcup I_2 \sqcup \cdots \sqcup I_{t+1}]$ is also a semi-positive \zt-tensor. Thus by induction,
 we can prove that the weakly reducible nonsingular \mt-tensor $\ten{A}$ is a semi-positive \zt-tensor as well.
\end{proof}

Combining Lemma \ref{lemma4} and Lemma \ref{lemma5}, we finish the
proof of  Theorem \ref{thm1}. Thus all the properties of
semi-positive \zt-tensors we investigate above are the same with
nonsingular \mt-tensors, and vice versa. The semi-positivity can be
employed to study the nonsingular \mt-tensors afterwards.

\subsection{General \mt-Tensors}

We discuss the nonsingular \mt-tensors above, moreover the general
\mt-tensors are also useful. The examples can be found in the
literature. For instance, the Laplacian tensor $\ten{L}$ of a
hypergraph (cf. \cite{Hu13a,Hu13b,Hu13c,Qi13}) is an \mt-tensor but
is not a nonsingular \mt-tensor.

An \mt-tensor can be written as $\ten{A} = s\ten{I} - \ten{B}$,
where $\ten{B}$ is nonnegative and $s \geq \rho(\ten{B})$. It is
easy to verify that the tensor $\ten{A}_\varepsilon = \ten{A} +
\varepsilon \ten{I}$ ($\varepsilon > 0$) is then a nonsingular
\mt-tensor and $\ten{A}$ is the limitation of a series of
$\ten{A}_\varepsilon$ when $\varepsilon \rightarrow 0$. Since all
the diagonal entries of a semi-positive \zt-tensor, i.e., a
nonsingular \mt-tensor, are positive, therefore the diagonal entries
of a general \mt-tensor, as the limitation of a series of
nonsingular \mt-tensor, must be nonnegative. Thus we prove the
following proposition.

\begin{proposition}
A general \mt-tensor has all nonnegative diagonal entries.
\end{proposition}

The conception semi-positivity \cite{Berman94} can be extended as
follows.
\begin{definition}[Semi-nonnegative tensor]
We call a tensor $\ten{A}$ as a \emph{semi-nonnegative tensor}, if there exists $x > 0$ such that $\ten{A} x^{m-1} \geq 0$.
\end{definition}

Not like the semi-positive case, a tensor is a semi-nonnegative
\zt-tensor is not equivalent to that it is a general \mt-tensor.
Actually, a semi-nonnegative \zt-tensor must be an \mt-tensor, but
the converse is not true. The proof of the first statement is
analogous to Lemma \ref{lemma4}, so we have the next theorem.
\begin{theorem}
A semi-nonnegative \zt-tensor is also an \mt-tensor.
\end{theorem}
Conversely, we can give a counterexample to show that there exists an \mt-tensor which is not semi-nonnegative. Let $\ten{B}$ be a nonnegative tensor of size $2 \times 2 \times 2 \times 2$, and the entries are as follows:
$$
b_{1111} = 2,\ b_{1122} = b_{2222} = 1,\ \text{and } b_{i_1 i_2 i_3 i_4} = 0 \text{ otherwise}.
$$
Then the spectral radius of $\ten{B}$ is apparently $2$. Let $\ten{A} = 2 \ten{I} - \ten{B}$, then $\ten{A}$ is an \mt-tensor with entries
$$
a_{1122} = -1,\ a_{2222} = 1,\ \text{and } a_{i_1 i_2 i_3 i_4} = 0 \text{ otherwise}.
$$
Thus for every positive vector $x$, the first component of $\ten{A}x^3$ is always negative and the second one is positive, which is to say that there is no such a positive vector $x$ with $\ten{A}x^3 \geq 0$.
Therefore $\ten{A}$ is an \mt-tensor but is not semi-nonnegative. However there are still some special \mt-tensors are semi-nonnegative.
\begin{proposition}
The following tensors are semi-nonnegative:
\begin{enumerate}[\rm (1)]
  \item A diagonally dominant \zt-tensor with nonnegative diagonal entries is semi-nonnegative;
  \item A weakly irreducible \mt-tensor is semi-nonnegative;
  \item Let $\ten{A} = s\ten{I} - \ten{B}$ be a weakly reducible \mt-tensor, where $\ten{B} \geq 0$ and $s = \rho(\ten{B})$, and $I_1,I_2,\cdots,I_k$ be the same as in Lemma \ref{lemma5}. If
      $$
      \rho(\ten{B}[I_t])
      \left\{
      \begin{array}{ll}
       < s, & t = 1,2,\cdots,k_1, \\
       = s, & t = k_1+1,k_1+2,\cdots,k
      \end{array}
      \right.
      $$
      and the entries of $\ten{B}[I_{k_1+1} \sqcup  I_{k_1+2} \sqcup  \cdots \sqcup I_k]$ are all zeros except the ones in the leading subtensors $\ten{B}[I_{k_1+1}], \ten{B}[I_{k_1+2}],\cdots,\ten{B}[I_k]$,
      then $\ten{A}$ is semi-nonnegative;
  \item A symmetric \mt-tensor is semi-nonnegative.
\end{enumerate}
\end{proposition}
The proofs of (1)$\sim$(3) are similar with the semi-positive case and (4) is a direct corollary of (3), therefore we omit them.


\section{Monotonicity}
Following the condition (C9) in Section \ref{mmatrix}, we generalize
\emph{monotone} \cite{Berman94} from nonsingular \textit{M}-matrices
to higher order tensors.

\subsection{Definitions}

\begin{definition}[Monotone tensor]
We call a tensor $\ten{A}$ as a \emph{monotone tensor}, if $\ten{A} x^{m-1} \geq 0$ implies $x \geq 0$.
\end{definition}

It is easy to show that the set of all monotone tensors is not empty, since the even order diagonal tensors with all positive diagonal entries belong to this set. However an odd order tensor is never a monotone tensor.
Since when $m$ is odd, $\ten{A} x^{m-1} \geq 0$ implies $\ten{A} (-x)^{m-1} \geq 0$ as well, thus we cannot guarantee that $x$ is nonnegative. Therefore we refer to even order tensors only in this section.

Sometimes we will use another equivalent definition of monotone tensor for convenience.

\begin{lemma}\label{lemma6}
An even order tensor $\ten{A}$ is a monotone tensor if and only if $\ten{A}x^{m-1} \leq 0$ implies $x \leq 0$.
\end{lemma}
\begin{proof}
Suppose that  $\ten{A}$ is a monotone tensor. Since $\ten{A}x^{m-1} \leq 0$ and $m-1$ is odd, we have
$$
\ten{A}(-x)^{m-1} = - \ten{A}x^{m-1} \geq 0.
$$
Then $-x \geq 0$ by the definition, which is equivalent to $x \leq 0$.

If $\ten{A}x^{m-1} \leq 0$ implies $x \leq 0$. When $\ten{A}y^{m-1} \geq 0$, we have
$$
\ten{A}(-y)^{m-1} = - \ten{A}y^{m-1} \leq 0.
$$
Therefore $-y \leq 0$, which is equivalent to $y \geq 0$. Thus $\ten{A}$ is a monotone tensor.
\end{proof}

\subsection{Properties}

We shall prove that a monotone \zt-tensor is also a nonsingular \mt-tensor. Before that, we need some lemmas.

\begin{lemma}\label{lemma7}
An even order monotone tensor has no zero H-eigenvalue.
\end{lemma}
\begin{proof}
Let $\ten{A}$ be an even order monotone tensor. If $\ten{A}$ has a zero H-eigenvalue, that is, there is a nonzero vector $x \in \mathbb{R}^n$ such that $\ten{A}x^{m-1} = 0$, then $\ten{A}(\alpha x)^{m-1} = 0$
as well for all $\alpha \in \mathbb{R}$. Thus we cannot ensure that $\alpha x \geq 0$, which contradicts the definition of a monotone tensor. Therefore $\ten{A}$ has no zero H-eigenvalue.
\end{proof}

\begin{lemma}\label{lemma8}
Every H$^+$-eigenvalue of an even order monotone tensor is nonnegative.
\end{lemma}
\begin{proof}
Let $\ten{A}$ be an even order monotone tensor and $\lambda$ be an H$^+$-eigenvalue of $\ten{A}$, that is, there is a nonzero vector $x \in \mathbb{R}_+^n$ such that $\ten{A}x^{m-1} = \lambda x^{[m-1]}$.
Then we have $\ten{A}(-x)^{m-1} = -\lambda x^{[m-1]}$, since $m-1$ is odd. If $\lambda < 0$, then $\ten{A}(-x)^{m-1} = -\lambda x^{[m-1]} \geq 0$, which indicates $-x \geq 0$ as well as $x \leq 0$.
It contradicts that $x$ is nonzero and nonnegative. Then $\lambda \geq 0$.
\end{proof}

The next theorem follows directly from Lemma \ref{lemma7} and Lemma \ref{lemma8}.
\begin{theorem}\label{thm9}
Every H$^+$-eigenvalue of an even order monotone tensor is positive.
\end{theorem}

By applying this result, we can now reveal the relationship of the
set of even order monotone \zt-tensors and that of nonsingular
\mt-tensors.

\begin{theorem}
An even order monotone \zt-tensor is also a nonsingular \mt-tensor.
\end{theorem}
\begin{proof}
Let $\ten{A}$ be an even order monotone \zt-tensor. We write $\ten{A} = s \ten{I} - \ten{B}$ since $\ten{A}$ is a \zt-tensor, where $\ten{B} \geq 0$. Then $\rho(\ten{B})$ is an H$^+$-eigenvalue of $\ten{B}$ by
Perron-Frobenius theorem for nonnegative tensors; that is, there is a nonzero vector $x \geq 0$ with $\ten{B}x^{m-1} = \rho(\ten{B}) x^{[m-1]}$. Then
$$
\ten{A}x^{m-1} = (s \ten{I} - \ten{B})x^{m-1} = (s-\rho(\ten{B})) x^{[m-1]},
$$
which is to say that $s - \rho(\ten{B})$ is an H$^+$-eigenvalue of $\ten{A}$. From Theorem \ref{thm9}, the H$^+$-eigenvalue $s - \rho(\ten{B}) > 0$, which indicates $s > \rho(\ten{B})$. So $\ten{A}$
is also a nonsingular \mt-tensor.
\end{proof}

This theorem tells us that
$$
\{\text{Monotone \zt-tensors of even order}\} \subseteq \{\text{Nonsingular \mt-tensors of even order}\}.
$$
However we will show that not every nonsingular \mt-tensor is monotone in the following subsection. The equivalent relation in matric situations between these two conditions is no longer true when the order is larger than $2$.

Next, we will present some properties of monotone \zt-tensors.

\begin{proposition}
An even order monotone \zt-tensor has all positive diagonal entries.
\end{proposition}
\begin{proof}
Let $\ten{A}$ be an even order monotone \zt-tensor. Consider $\ten{A} e_i^{m-1}$ ($i = 1,2,\cdots,n$), where $e_i$ denotes the vector with only one nonzero entry $1$ in the $i$-th position. We have
$$
\begin{array}{ll}
(\ten{A} e_i^{m-1})_i = a_{i i \cdots i}, & i = 1,2,\cdots,n \\
(\ten{A} e_i^{m-1})_j = a_{j i \cdots i} \leq 0, &  j \neq i.
\end{array}
$$
If $a_{i i \cdots i} \leq 0$ then $\ten{A} e_i^{m-1} \leq 0$, which indicates $e_i \leq 0$ by Lemma \ref{lemma6}, but it is impossible. So we have $a_{i i \cdots i} > 0$ for $i = 1,2,\cdots,n$.
\end{proof}

The next proposition shows some rows of a monotone \zt-tensor is
strictly diagonally dominant.

\begin{proposition}
Let $\ten{A}$ be an even order monotone \zt-tensor. Then
$$
a_{i i \cdots i} > \displaystyle\sum_{(i_2,i_3,\cdots,i_m) \neq
(i,i,\cdots,i)} |a_{i i_2 \cdots i_m}|
$$
for some $i \in \{1,2,\cdots,n\}$.
\end{proposition}
\begin{proof}
Consider $\ten{A} {\bf e}^{m-1}$ ($i = 1,2,\cdots,n$), where ${\bf
e}$ denotes the all ones vector. We have
$$
(\ten{A} {\bf e}^{m-1})_i = a_{i i \cdots i} +
\sum_{(i_2,i_3,\cdots,i_m) \neq (i,i,\cdots,i)} a_{i i_2 \cdots i_m}
= a_{i i \cdots i} - \sum_{(i_2,i_3,\cdots,i_m) \neq (i,i,\cdots,i)}
|a_{i i_2 \cdots i_m}|,
$$
since $a_{i i_2 \cdots i_m} \leq 0$ for $(i_2,i_3,\cdots,i_m) \neq
(i,i,\cdots,i)$.

If $a_{i i \cdots i} \leq \sum\limits_{(i_2,i_3,\cdots,i_m) \neq
(i,i,\cdots,i)} |a_{i i_2 \cdots i_m}|$ for all $i = 1,2,\cdots,n$,
then $\ten{A} {\bf e}^{m-1} \leq 0$, which indicates ${\bf e} \leq
0$ by Lemma \ref{lemma6}, and it is impossible. So $a_{i i \cdots i}
> \sum\limits_{(i_2,i_3,\cdots,i_m) \neq (i,i,\cdots,i)} |a_{i i_2
\cdots i_m}|$ for some $i$.
\end{proof}

\begin{proposition}
Let $\ten{A}$ be an even order monotone \zt-tensor. Then
$\ten{A}+\ten{D}$ has all positive H$^+$-eigenvalues for each
nonnegative diagonal tensor $\ten{D}$.
\end{proposition}
\begin{proof}
If $\ten{A}+\ten{D}$ has a non-positive H$^+$-eigenvalue, that is,
there exists a nonzero vector $x \geq 0$ such that
$(\ten{A}+\ten{D})x^{m-1} = \lambda x^{[m-1]}$ and $\lambda \leq 0$,
then $\ten{A}x^{m-1} = \lambda x^{[m-1]} - \ten{D}x^{m-1} \leq 0$,
since $x$ and $\ten{D}$ are nonnegative and $\lambda$ is
non-positive, which implies $x \leq 0$ from the definition of
monotone \zt-tensors. This is a contradiction. Therefore
$\ten{A}+\ten{D}$ has no non-positive H$^+$-eigenvalues for all
nonnegative diagonal tensor $\ten{D}$.
\end{proof}

Furthermore, the monotone \zt-tensor also has the following two
equivalent definitions, following the condition (C10) and (C11).
\begin{proposition}
A \zt-tensor $\ten{A}$ is monotone if and only if there exists a positive diagonal tensor $\ten{D}$ and a monotone \zt-tensor $\ten{C}$
such that $\ten{A} = \ten{D}\ten{C}$.
\end{proposition}
\begin{proof}
Let $\ten{D}$ be the diagonal tensor of $\ten{A}$ and $\ten{C}=\ten{D}^{-1}\ten{A}$. Clearly, $\ten{A} = \ten{D}\ten{C}$.

If $\ten{A}$ is a monotone \zt-tensor, then $\ten{D}$ is positive diagonal and $\ten{A}x^{m-1} \geq 0$ implies $x \geq 0$. When $\ten{C}x^{m-1} \geq 0$, the vector
$\ten{A}x^{m-1} = \ten{D} \Big((\ten{C} x^{m-1})^{\left[\frac{1}{m-1}\right]}\Big)^{m-1}$ is also nonnegative, thus $x \geq 0$. Since $\ten{C}$ is also a \zt-tensor, then $\ten{C}$ is a monotone \zt-tensor.

If $\ten{C}$ is a monotone \zt-tensor and $\ten{D}$ is positive diagonal, then $\ten{C}x^{m-1} \geq 0$ implies $x \geq 0$. When $\ten{A}x^{m-1} \geq 0$, the vector
 $\ten{C}x^{m-1} = \ten{D}^{-1} \Big((\ten{A} x^{m-1})^{\left[\frac{1}{m-1}\right]}\Big)^{m-1}$ is also nonnegative, thus $x \geq 0$. So $\ten{A}$ is a monotone \zt-tensor.
\end{proof}

\begin{proposition}
A \zt-tensor $\ten{A}$ is monotone if and only if there exists a
positive diagonal tensor $\ten{D}$ and a nonnegative tensor
$\ten{E}$ such that $\ten{A} = \ten{D} - \ten{E}$ and
$\ten{D}^{-1}\ten{E}$ satisfying $(\ten{D}^{-1}\ten{E})x^{m-1} \leq
x^{[m-1]}$ implies $x \geq 0$.
\end{proposition}
\begin{proof}
Let $\ten{D}$ be the diagonal tensor of $\ten{A}$ and $\ten{E}=\ten{D}-\ten{A}$. Clearly, $\ten{A} = \ten{D}-\ten{E}$ and $\ten{D}^{-1}\ten{E} = \ten{I} - \ten{D}^{-1}\ten{A}$.

If $\ten{A}$ is a monotone \zt-tensor, then $\ten{D}$ is positive diagonal and $\ten{A}x^{m-1} \geq 0$ implies $x \geq 0$. When $(\ten{D}^{-1}\ten{E})x^{m-1} \leq x^{[m-1]}$, we have
$\ten{E}x^{m-1} \leq \ten{D}x^{m-1}$, and thus $\ten{A} x^{m-1} \geq 0$. This indicates $x \geq 0$.

If $(\ten{D}^{-1}\ten{E})x^{m-1} \leq x^{[m-1]}$ implies $x \geq 0$. When $\ten{A} x^{m-1} \geq 0$, we have $\ten{D}x^{m-1} \geq \ten{E}x^{m-1}$, and thus $(\ten{D}^{-1}\ten{E})x^{m-1} \leq x^{[m-1]}$,
which indicates $x \geq 0$. Since $\ten{A}$ is also a \zt-tensor, then $\ten{A}$ is a monotone \zt-tensor.
\end{proof}

\subsection{A Counterexample}

We will give a $4$-order counterexample in this section to show that the set of all monotone \zt-tensor is a proper subset of the set of all nonsingular \mt-tensor, when the order is larger than $2$.

The denotation of the Kronecker product \cite{Golub13} for $\ten{A} = X \otimes Y$ means $a_{i_1 i_2 i_3 i_4} = x_{i_1 i_2} \cdot y_{i_3 i_4}$. Let $\ten{J} = I_n \otimes I_n$, where $I_n$ denotes the $n \times n$
 identity matrix. It is obvious that the spectral radius $\rho(\ten{J}) = n$, since the sum of each rows of $\ten{J}$ equals $n$. Take
$$
\ten{A} = s\ten{I} - \ten{J}\ (s>n)
\text{  and  }
x = \begin{bmatrix} 1 \\ \vdots \\ 1 \\ -\delta \end{bmatrix}\ (0 < \delta < 1).
$$
Then $\ten{A}$ is a nonsingular \mt-tensor and
$$
\ten{A} x^3 = s x^{[3]} - (x^{\rm T} x) x = s\begin{bmatrix} 1 \\ \vdots \\ 1 \\ -\delta^3 \end{bmatrix} - (n - 1 + \delta^2) \begin{bmatrix} 1 \\ \vdots \\ 1 \\ -\delta \end{bmatrix}
= \begin{bmatrix} s-n+1-\delta^2 \\ \vdots \\ s-n+1-\delta^2 \\ (n - 1 +(1-s) \delta^2)\delta \end{bmatrix}.
$$
When $\delta \leq \sqrt{\frac{n-1}{s-1}}$, the vector $\ten{A} x^3$
is nonnegative while $x$ is not nonnegative. Therefore $\ten{A}$ is
not a monotone \zt-tensor, although it is a nonsingular \mt-tensor.

\subsection{An Example}

We now give an example of nontrivial monotone \zt-tensor also
applying the Kronecker product. Let $\ten{B} = (a^{[2k-1]}b^{\rm T})
\otimes (bb^{\rm T}) \cdots \otimes (bb^{\rm T})$ be a tensor of
order $2k$, where $a$ and $b$ are nonnegative vectors. It is direct
to compute that $\rho(\ten{B}) = (b^{\rm T} a)^{2k-1}$. Then
$\ten{A} = s\ten{I} - \ten{B}$ ($s>(b^{\rm T} a)^{2k-1}$) is a
nonsingular \mt-tensor. For each $x$, we get
$$
\ten{A}x^{2k-1} = s x^{[2k-1]} - a^{[2k-1]}(b^{\rm T} x)^{2k-1}.
$$
When $\ten{A}x^{2k-1} \geq 0$, we have
$$
s^{\frac{1}{2k-1}} \cdot x \geq a \cdot (b^{\rm T} x),
$$
thus
$$
s^{\frac{1}{2k-1}} \cdot (b^{\rm T} x) \geq (b^{\rm T} a) (b^{\rm T} x).
$$
Since $s>(b^{\rm T} a)^{2k-1}$, we can conclude $b^{\rm T} x \geq
0$. So $x \geq a \cdot \frac{(b^{\rm T} x)}{s} \geq 0$, which
indicates that $\ten{A}$ is a monotone \zt-tensor.


\section{An Extension of \mt-Tensors}

Inspired by the conception of \textit{H}-matrix  \cite{Berman94}, we
can extend \mt-tensors to \htt-tensors. First, we define the
comparison tensor.
\begin{definition}
Let $\ten{A} = (a_{i_1 i_2 \cdots i_m})$ be a tensor of order $m$ and dimension $n$. We call another tensor $\ten{M}(\ten{A}) = (m_{i_1 i_2 \cdots i_m})$ as the \emph{comparison tensor} of $\ten{A}$ if
$$
m_{i_1 i_2 \cdots i_m} =
\left\{
\begin{array}{ll}
+|a_{i_1 i_2 \cdots i_m}|, & \text{if } (i_2,i_3,\cdots,i_m) = (i_1,i_1,\cdots,i_1), \\
-|a_{i_1 i_2 \cdots i_m}|, & \text{if } (i_2,i_3,\cdots,i_m) \neq
(i_1,i_1,\cdots,i_1).
\end{array}
\right.
$$
\end{definition}

Then we can state what is an \htt-tensor.
\begin{definition}
We call a tensor an \emph{\htt-tensor}, if its comparison tensor is
an \mt-tensor; And we call it as a \emph{nonsingular \htt-tensor},
if its comparison tensor is a nonsingular \mt-tensor.
\end{definition}

Nonsingular \htt-tensors has a property called quasi-strictly diagonally dominant, which can be proved directly from the properties of nonsingular \mt-tensors. Therefore we will omit the proof.
\begin{proposition}
A tensor $\ten{A}$ is a nonsingular \htt-tensor if and only if it is \emph{quasi-strictly diagonally dominant}, that is, there exist $n$ positive real numbers $d_1,d_2,\cdots,d_n$ such that
$$
|a_{i i \cdots i}| d_i^{m-1} > \sum_{(i_2,i_3,\cdots,i_m) \neq
(i,i,\cdots,i)} |a_{i i_2 \cdots i_m}| d_{i_2} \cdots d_{i_m},\ i =
1,2,\cdots,n.
$$
\end{proposition}

Similarly to the nonsingular \mt-tensor, nonsingular \htt-tensor has other equivalent definitions.
\begin{proposition}
The following conditions are equivalent:
\begin{enumerate}[\rm ({E}1)]
  \item A tensor $\ten{A}$ is a nonsingular \htt-tensor;
  \item There exists a positive diagonal matrix $D$ such that $\ten{A}D^{m-1}$ is strictly diagonally dominant;
  \item There exist two positive diagonal matrix $D_1$ and $D_2$ such that $D_1\ten{A}D_2^{m-1}$ is strictly diagonally dominant.
\end{enumerate}
\end{proposition}


\section{Conclusions}

In this paper, we give the proofs or the counterexamples to show
these three relations between different sets of tensors:
\[
\begin{split}
\{\text{semi-positive \zt-tensors}\} &= \{\text{nonsingular \mt-tensors}\}, \\
\{\text{semi-nonnegative \zt-tensors}\} &\subsetneqq \{\text{general \mt-tensors}\}, \\
\{\text{even-order monotone \zt-tensors}\} &\subsetneqq \{\text{even-order nonsingular \mt-tensors}\}.
\end{split}
\]
Applying these relations, we investigate the properties of
nonsingular and general \mt-tensors. Along with the results in Zhang
el al. \cite{Zhang12}, the equivalent conditions of nonsingular
\mt-tensors until now are listed as follows.\\

 If $\ten{A}$ is a
\zt-tensor, then the following conditions are equivalent:
\begin{enumerate}[({D}1)]
  \item $\ten{A}$ is a nonsingular \mt-tensor;
  \item Every real eigenvalue of $\ten{A}$ is positive; (\cite{Zhang12})
  \item The real part of each eigenvalue of $\ten{A}$ is positive; (\cite{Zhang12})
  \item $\ten{A}$ is semi-positive; that is, there exists $x > 0$ with $\ten{A} x^{m-1} > 0$;
  \item There exists $x \geq 0$ with $\ten{A} x^{m-1} > 0$;
  \item $\ten{A}$ has all positive diagonal entries, and there exists a positive diagonal matrix $D$ such that $\ten{A}D^{m-1}$ is strictly diagonally dominant;
  \item $\ten{A}$ has all positive diagonal entries, and there exist two positive diagonal matrices $D_1$ and $D_2$
  such that $D_1 \ten{A} D_2^{m-1}$ is strictly diagonally dominant;
  \item There exists a positive diagonal tensor $\ten{D}$ and a nonsingular \mt-tensor $\ten{C}$ with $\ten{A} = \ten{D}\ten{C}$;
  \item There exists a positive diagonal tensor $\ten{D}$ and a nonnegative tensor $\ten{E}$ such that $\ten{A} = \ten{D} - \ten{E}$ and there exists $x>0$ with $(\ten{D}^{-1}\ten{E})x^{m-1} < x^{[m-1]}$.
\end{enumerate}

\end{document}